\documentclass[12pt,reqno]{amsart}
\usepackage{amsmath,amssymb,amsthm,mathtools,calc,verbatim,enumitem,tikz,url,hyperref,mathrsfs,cite,fullpage}
\usepackage{bbm}
\usepackage{stmaryrd}
\usepackage{textcomp}
\usepackage{setspace}
\usepackage{amssymb}
\usepackage{amsthm}
\usepackage{amsmath}
\usepackage{graphicx}
\usepackage{marvosym}
\usepackage{empheq}
\usepackage{latexsym}
\usepackage{fontenc}
\usepackage{color}

\newcommand\R{\mathbb{R}}

\renewcommand\le{\leqslant}
\renewcommand\ge{\geqslant}
\renewcommand\to{\rightarrow}

	\def\R{\mathbb{R}}

	\def\<{\langle }
	\def\>{\rangle }

\usepackage{amssymb}
\newtheorem{theor}{Theorem}[section]
\newtheorem{prop}[theor]{Proposition}
\newtheorem{lemma}[theor]{Lemma}
\newtheorem{conj}[theor]{Conjecture}
\theoremstyle{definition}
\newtheorem{defi}[theor]{Definition}
\theoremstyle{remark}
\newtheorem{rema}[theor]{Remark}
\newtheorem{exam}[theor]{Example}

\newtheorem*{claim*}{Claim}
\newtheorem*{qu*}{Question}

\newcommand{\Zdd}{\mathbb Z^2}

\newcommand{\Zd}{\mathbb Z^d}

\newcommand{\p}{\mathbb P}
\newcommand{\e}{\mathbb E}
\newcommand{\U}{\mathcal U}
\newcommand{\UU}{\mathcal U}
\newcommand{\dhp}{\mathbb H_u}
\newcommand{\Ss}{\mathcal{S}}

\usepackage{dsfont}
\usepackage{caption}
\usepackage{subcaption}
\usepackage{pifont}
\begin{document}

\title{Fixation for two-dimensional $\UU$-Ising and $\UU$-voter dynamics}
\author{Daniel Blanquicett}
\address{Mathematics Department,
University of California, Davis, CA 95616, USA}
\email{drbt@math.ucdavis.edu}
\thanks{{\it Date}: March 4, 2020.
\indent The author was partially supported by CAPES, Brasil.}


	
\begin{abstract}
Given a finite family $\U$ of finite subsets of $\Zd\setminus \{0\}$,
the $\U${\it -voter dynamics} in the space of configurations $\{+,-\}^{\Zd}$ is defined
as follows: every $v\in\Zd$ has an independent exponential random clock, and
when the clock at $v$ rings, the vertex $v$ chooses $X\in\U$ uniformly at random. If the set $v+X$ is entirely in state $+$ (resp. $-$), then the state of $v$ updates to $+$ (resp. $-$), otherwise nothing happens.
The {\it critical probability} $p_c^{\text{vot}}(\Zd,\U)$ for this model is the infimum over $p$ such that this system
almost surely fixates at $+$ when the initial states for the vertices are chosen independently to be $+$ with
probability $p$ and to be $-$ with probability $1-p$.
We prove that $p_c^{\text{vot}}(\Zdd,\U)<1$ for a wide class of families $\U$.

We moreover consider the $\UU$-Ising dynamics and show that this model also exhibits the same phase transition.

\end{abstract}
	
\maketitle 
\section{Introduction}
  Given some spin dynamics on $ \Zd$, the {\it critical probability} for fixation is the infimum over
$p\in[0,1]$ such that fixation at $+$ occurs almost surely when the initial states for the vertices are chosen independently
to be $+$ with probability $p$ and to be $-$ with probability $1-p$.
For the zero-temperature Glauber dynamics of the Ising model, Fontes, Schonmann and Sidoravicius \cite{FSS02} showed that
$p_c^{\text{Is}}(\Zd)<1$ (Theorem \ref{VladasT}). In other words, there exists a {\it phase transition}, since by symmetry between $+$ and $-$, $p_c^{\text{Is}}(\Zd)\ge 1/2$.

In recent groundbreaking work, Bollob\'as, Smith and Uzzell \cite{BSU15} introduced the $\U$-bootstrap percolation model (see Section \ref{bpf}),
where $\U$ 
is a finite family of finite
subsets of $\Zd\setminus \{0\}$, which motivated Morris \cite{Morris17} to generalize the Glauber dynamics by defining the $\UU$-Ising dynamics (see Section \ref{Glauber}); he conjectured that 
for the so called {\it critical} families,
this model also exhibits a phase transition.
In this note we prove that this conjecture is true under suitable conditions.
We also consider a variant of these dynamics that we call the  $\UU$-voter dynamics (see Section \ref{Uvoter}), and show that 
in this case, for a wide class of critical families, we also have a phase transition.

\subsection{The $\UU$-Ising dynamics}\label{Glauber}
Let $\UU=\{X_1,\dots,X_m\}$ be an arbitrary finite family of finite subsets of $\Zd\setminus \{0\}$. 
Given a configuration in $\{+,-\}^{\Zd}$, we say that $X\in\UU$ 
{\it disagrees with vertex} $v\in\Zd$ if each vertex in $v+X$ 
 has the opposite state to that of $v$.
The $\UU${\it -Ising dynamics} on $\Zd$ with states $+$ and $-$ were introduced by Morris \cite{Morris17}
as follows:
 \begin{itemize}
 \item Every $v\in\Zd$ has an independent exponential random clock with rate 1.
  \item When the clock at vertex $v$ rings at (continuous) time $t\ge 0$, 
  if {\it there exists} $X\in\UU$ which disagrees with $v$, then $v$ flips its state. Otherwise nothing happens.
 \end{itemize}
 
We are interested in the long-term behavior of this system, starting from a randomly chosen initial state,  and ask whether the dynamics fixate or not.

Special cases of these dynamics have been extensively studied, 
for example,  consider the family $\mathcal N_r^d$ defined as
the collection of all subsets of size $\ge r$ of $\{\pm e_1,\dots,\pm e_d\}$;
when $r=d$ this process coincides with the so called
{\it zero-temperature Glauber dynamics of the Ising model} (sometimes called {\it Metropolis dynamics}), see, for example \cite{Martinelli99}.

Let $\sigma_t\in\{+,-\}^{\Zd}$ denote the state of the system at time $t\ge 0$.
Say that dynamics {\it fixate at} + if for each vertex $v\in\Zd$, there is a time
$T_v\in[0,\infty)$ such that $\sigma_t(v)=+$ for all $t\ge T_v$, in other words, if the state of
every vertex is eventually $+$.
Now fix $p\in[0,1]$; we say that a set $A\subset\Zd$ is $p${\it-random} if 
it is chosen according to the Bernoulli product measure on $\Zd$ 
(i.e. each of the sites of $\Zd$ are included in $A$ independently with probability $p$).
Let the set $\{v\in\Zd:\sigma_0(v)=+\}$ be chosen $p$-randomly and
write $\p_p$ for the joint distribution of 
the initial spins and the dynamics realizations.
 We define the {\it critical probability} for the $\U$-Ising dynamics to be
 \[p_c^{\text{Is}}(\Zd,\U):=\inf\left\{p:\p_p(\text{$\U$-Ising dynamics fixate at }+)=1\right\},\]
and write $p_c^{\text{Is}}(\Zd)$ for $p_c^{\text{Is}}(\Zd,\mathcal N_d^d)$.
Arratia \cite{Arratia83} proved that $p_c^{\text{Is}}(\mathbbm Z)=1$, 
and moreover that, for every $p\in(0,1)$, 
every site changes state an infinite number of times.
A well-known (and possibly folklore) conjecture states that
$p_c^{\text{Is}}(\Zd)=1/2$ for every $d\ge 2$.
The first progress towards this conjecture was the following upper bound, proved by
Fontes, Schonmann and Sidoravicius \cite{FSS02}.
\begin{theor}[Fontes, Schonmann and Sidoravicius]\label{VladasT}
$p_c^{\textup{Is}}(\Zd)<1$ for every $d\ge 2$. 
\end{theor}
Moreover, the authors in \cite{FSS02} showed that this fixation occurs in time with a stretched exponential tail.
Morris \cite{Morris09} combined this theorem with techniques from high dimensional
bootstrap percolation (see Section \ref{bpf}) to prove that $p_c^{\text{Is}}(\Zd)\to 1/2$ as $d\to\infty$.

Another related result for the symmetric case $p=1/2$ (which corresponds to an initial quench from infinite temperature) is due to Nanda, Newman and Stein \cite{NNS00}.
They proved that in two dimensions, every vertex almost surely changes state an infinite number of times; however, it is still unknown if the same holds for higher dimensions.

These dynamics have also been considered in other lattices. For instance, 
Damron, Kogan, Newman and Sidoravicius \cite{DKNS13} considered slabs of the form
$\mathbbm S_k:=\Zdd\times\{0,1,\dots,k-1\}$ ($k\ge 2$) with the family $\mathcal N_3^3$.
They proved a classification theorem which, surprisingly, holds for {\it all} $p\in(0,1)$
and, in particular implies that $\mathbbm S_k$ does not fixate at $+$ (however, each single
vertex in $\mathbbm S_2$ fixates at either $+$ or $-$).
Therefore, in this particular setting, which interpolates between dimensions 2 and 3 (so Theorem \ref{VladasT} does not apply),
the critical probability is 1 and there is no phase transition for fixation at $+$.

Let us now consider a general family $\U$ in dimension $d=2$.
For each $u\in S^1$ (the unit circle) we write $\dhp:=\{x\in\Zdd:\langle x,u\rangle<0\}$ 
 for the discrete half-plane whose boundary is perpendicular to $u$.
 In their groundbreaking work on general models of monotone cellular automata,
 Bollob\'as, Smith and Uzzell \cite{BSU15}  made the following important definitions. 
\begin{defi}\label{crit} The set $\Ss$ of {\it stable directions} is
\[\Ss=\Ss(\U):=\{u\in S^1: X\not\subset\mathbbm H_u,\ \forall X\in\U\}.\]
We say that $\U$ is {\it critical} if there exists a semicircle in $S^1$ that has finite intersection with $\Ss$,
and if every open semicircle in $S^1$ has non-empty intersection with $\Ss$. 
\end{defi}
For example, the family $\U=\mathcal N_2^2$ is critical, since $\mathcal S(\mathcal N_2^2)=\{\pm e_1,\pm e_2\}$.
Morris \cite{Morris17} conjectured that some of the known results about the family $\mathcal N_2^2$ can be extended to the general setting of critical models. For instance, the existence of the phase transition proved by Fontes, Schonmann and Sidoravicius \cite{FSS02}; it has been conjectured that such a transition is sharp and occurs at $p=1/2$. Moreover, the same result proved by Nanda, Newman and Stein \cite{NNS00} should hold. More precisely, he conjectured the following.
\begin{conj}\label{Co2}
For every critical two-dimensional family $\UU$, it holds that
\[p_c^{\textup{Is}}(\Zdd,\UU)<1.\]
\end{conj}

\begin{conj}
If $\U$ is a critical two-dimensional family and $p=1/2$, then almost
surely every vertex changes state an infinite number of times.
\end{conj}

In this note, we prove that Conjecture \ref{Co2} holds for a subclass of critical families.
\begin{defi}\label{DefiDrop}
 Let $\mathcal T\subset \Ss$. A $\mathcal T$-{\it droplet} is a non-empty set of the form
 \[D=\bigcap_{u\in\mathcal T}(\dhp+a_u),\]
 for some collection $\{a_u\in\Zdd:u\in\mathcal T\}$. When $D$ is finite and its diameter (the maximum distance between two points in $D$)
 is $\le L$, we call $D$ a $(\mathcal T,L)$-{\it droplet}.
\end{defi}
We will always consider subsets $\mathcal T\subset \Ss$ 
such that $D$ is finite (for example, when $\U$ is critical we can choose at least one such $\mathcal T$).
Suppose for a moment that every vertex in a $(\mathcal T,L)$-droplet $D$ is
in state $-$,
and every vertex
outside $D$ is frozen in state $+$ (see Figure \ref{frozen});
when we run the $\U$-Ising dynamics, one might expect $D$ to become entirely filled with $+$ in polynomial time in $L$.
\begin{defi}\label{TdD}
 Let $D$ be a $(\mathcal T,L)$-droplet. 
Assume we start the process with
$D$ entirely occupied by states $-$, and all other states are $+$.
The {\it droplet erosion time} $T^{\textup{Is}}(D)$ is the first time when $D$ is fully $+$.
\end{defi}
\begin{figure}[ht]
	\centering
	\includegraphics[scale=.6]{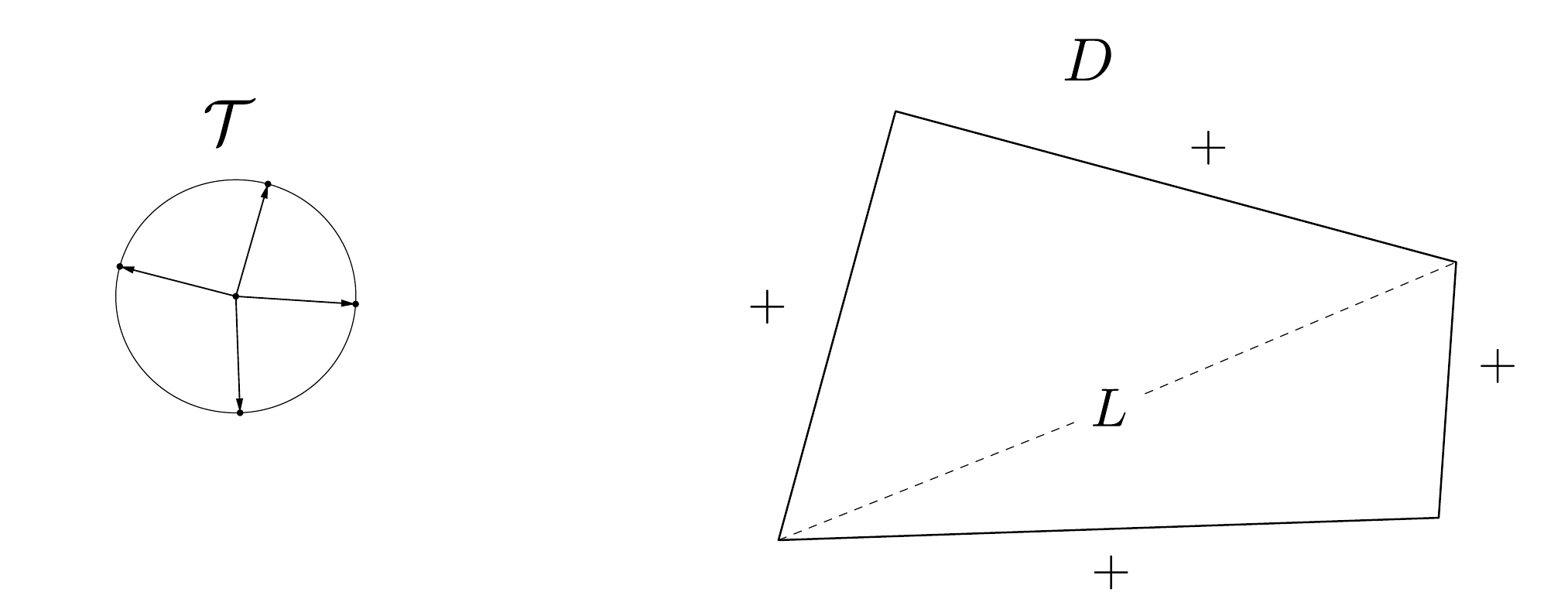}
	\caption{4 stable directions determining a $(\mathcal T,L)$-droplet.}
	\label{frozen}
\end{figure}

The droplet erosion time is well defined, because $\mathcal T\subset\Ss$, so
the states outside $D$ will never flip (see Figure \ref{frozen}), and eventually every state in $D$ will become $+$ forever.
\begin{defi}\label{polyerodii}
We say that $\UU$ is {\it Ising-eroding} if we can choose a constant $c>1$ and a finite set $\mathcal T\subset\Ss$, such that
 any $(\mathcal T,L)$-droplet $D$ satisfies
 \begin{equation}
  \p_p(T^{\textup{Is}}(D)> L^c)\le e^{- L},
 \end{equation}
 for all $L$ large enough. 
\end{defi}
The authors of \cite{FSS02} proved that $\mathcal N_2^2$ is $(2+\varepsilon,\mathcal S)$-eroding, for any fixed constant $\varepsilon>0$ and $\mathcal S=\Ss(\mathcal N_2^2)=\{\pm e_1, \pm e_2\}$.  Indeed, they proved that
\[ \p_p(T^{\textup{Is}}(D)> CL^2)\le e^{-\gamma L},\]
for some positive constants $C$ and $\gamma$, and all $L$ large enough.  Moreover, numerical simulations suggest the following conjecture to be true, which  seems hard to prove.
\begin{conj}\label{mia}
 Every critical family is Ising-eroding.
\end{conj}

Our main theorem states that 
there exists a phase transition for {\it some} critical families
(Conjecture \ref{mia} would imply it for {\it all} critical families).

\begin{theor}\label{Main2}
If $\UU$ is a Ising-eroding critical two-dimensional family, then
 \begin{equation}
  p_c^{\textup{Is}}(\Zdd,\UU)<1.
 \end{equation}
\end{theor}

We do not know how to prove that $(\mathcal T,L)$-droplets can be eroded in polynomial time in the sense of Definition \ref{polyerodii}.
For this reason we instead focus on the $\UU$-voter model where, as we will see, there is an additional bias in favor of the leading state that we will be able to exploit (see Proposition \ref{erosion1}).

\subsection{The $\UU$-voter dynamics}\label{Uvoter}
\begin{defi}\label{Uvoterdyn}
Let $\UU=\{X_1,\dots,X_m\}$ be an arbitrary finite family of finite subsets of $\Zd\setminus \{0\}$.
 The {\it $\UU$-voter dynamics} on $\Zd$ with states $+$ and $-$ are defined as follows:
 \end{defi}
 \begin{enumerate}
  \item[(a)] Every $v\in\Zd$ has an independent exponential random clock with rate 1.
  \item[(b)] When the clock at $v$ rings, 
  the vertex $v$ {\it chooses} $X\in\UU$ uniformly at random. If the set $v+X$ is entirely in state $+$ (resp. $-$), then the state of $v$ updates to $+$ (resp. $-$). Otherwise nothing happens.
 \end{enumerate}
 
Observe that in this case, the rule $X\in\U$ is chosen at random with probability $1/m$, this is the difference between the $\U$-Ising and $\U$-voter dynamics.

For example, when $\UU=\UU(V):=\{\{x\}:x\in V\}$ for some finite set $V\subset \Zd\setminus\{0\}$, in (b)
the vertex $v$ chooses some $x\in V$ independently with probability $1/|V|$, and then
vertex $v$ immediately adopts the same state as $x$; this is usually called a
{\it linear voter model}. 
Of particular interest is the case where $V$ consists of all $2d$ unit vectors in $\Zd$. For related results see \cite{Liggett99} and references therein.

The generator $\mathbb V$ of this Markov process acts on local functions $f$ as
\[\mathbb V f(\sigma)=\sum_{v\in\Zd}\frac{r_v(\sigma)}{m}[f(\sigma^v)-f(\sigma)],\]
here $r_v(\sigma)$ denotes the number of rules disagreeing with vertex $v$ when the current configuration
is $\sigma$.
Observe that we have symmetry with respect to the interchange of the roles of $-$s and $+$s for these dynamics, and the system is monotone, namely,
$r_v(\sigma)$ is increasing in $\sigma$ when $\sigma(v)=-$ and decreasing in $\sigma$ when $\sigma(v)=+$.

Let $p_c^{\textup{vot}}(\Zd,\UU)$  be the {\it critical probability} of the $\UU$-voter dynamics on $\Zd$
\begin{equation}
 p_c^{\textup{vot}}(\Zd,\UU):=\inf\left\{p:\p_p(\text{$\UU$-voter dynamics fixate at }+)=1\right\}.
\end{equation}

 We remark that the families $\UU(V)$ described above are not critical and, in fact,
 their dynamics do not fixate at $+$ (unless $p=1$). For instance, if $V$ consists of all $2d$ unit vectors and $d\ge 2$, then almost surely
 \[\frac{\int_0^t{\bf 1}\{\sigma_s(v)=-\}\,ds}{t}
\to 1-p,\] 
 as $t\to\infty$ (see \cite{CG83});
 but if fixation at $+$ occurred then this ratio should converge to $0$.
 However, critical families exhibit a behavior substantially different from that of $\UU(V)$.

Now, given a $(\mathcal T,L)$-droplet $D$, assume that we start the $\U$-voter dynamics with $D$ entirely occupied by states $-$, and all other states are $+$. The {\it voter erosion time} $T(D)$ is the first time when $D$ is fully $+$.
\begin{defi}\label{chuloenbola}
We say that $\UU$ is {\it voter-eroding} if there exist $c>1$ and $\mathcal T\subset\Ss$, such that
 any $(\mathcal T,L)$-droplet $D$ satisfies
 \begin{equation}\label{polyerod}
  \p_p(T(D)> L^c)\le e^{- L},
 \end{equation}
 for all $L$ large enough. We say moreover that $\UU$ is $(c,\mathcal T)$-{\it eroding}.
\end{defi}
 The following is our main result.
\begin{theor}\label{Main}
If $\UU$ is a voter-eroding critical two-dimensional family, then
 \begin{equation}
  p_c^{\textup{vot}}(\Zdd,\UU)<1.
 \end{equation}
\end{theor}
The proof of this theorem is essentially equivalent to the proof of Theorem \ref{Main2}, thus, from now on, we will only focus on the $\U$-voter dynamics. The only advantage is that in this case, we can provide explicit examples of voter-eroding critical families, like the following ones.
\begin{exam}\label{easyexamples}
\begin{enumerate}
    \item  $\U=
\{\{e_2,-e_2\},\{-e_1,e_2\},\{-e_1,-e_2\}\}$,  with  $\mathcal T=\{\pm e_1, \pm e_2\}$ (so that $\mathcal T$-droplets are rectangular).
This is usually called the {\it Duarte model} (see, e.g. \cite{M95}). 
    \item   $\U=
\{\{e_1,-3e_2, -2e_2, -e_2\},\ \{2e_1,3e_2, 5e_2\},\ \{-e_1,e_2\},\ \{-2e_1, -4e_2\},\ \{(1,1),-e_2\}\}$,  with  $\mathcal T=\Ss=\{\pm e_1, \pm e_2\}$, and
    \item $\U=\{\{(-1,1),(-1,-1)\},\ \{(0,1),(1,1)\},\ \{(0,-1),(1,-1)\},\ \{(-1,2),(-1,-1)\}\}$, with $\mathcal T=\Ss=\left\{-e_1, \frac{1}{\sqrt 2}(1,1), \frac{1}{\sqrt 2}(1,-1)\right\}$ (so, $\mathcal T$-droplets are triangular).
\end{enumerate}
\end{exam}

It is clear that the families in the above example are critical. In Section \ref{Exampless}, we will explain how to deduce that they are also voter-eroding, and other general sources of examples will be mentioned.

 

\section{Outline of the proof and bootstrap percolation}
\subsection{Outline of the proof}
Here we give a sketch of the proof of Theorem \ref{Main}. In order to prove Theorem \ref{Main}, we will combine techniques of \cite{BSU15} and \cite{FSS02}, indeed, we will be able to prove a stronger result, namely, that fixation occurs in time with a stretched exponential tail. From now on, all mentioned constants will depend on the family $\U$.

\begin{theor}\label{stronger}
Let $\UU$ be a voter-eroding critical two-dimensional family. There exist constants $\gamma>0$ and $p_0<1$ such that, for every $p>p_0$,
 \begin{equation}
  \p_p[\sigma_t(0)=-]\le \exp(-t^{\gamma}),
 \end{equation}
  for all sufficiently large $t$.
\end{theor}

Let us deduce Theorem \ref{Main} from Theorem \ref{stronger}.
\begin{proof}[Proof of Theorem \ref{Main}]
Fix $t\ge 0$ and consider the events
\[F:=\{\sigma_s(0)\textup{ is constant for }s\in[t-1,t]\},\]
\[F':=\{\exists s\in[t-1,t]:\sigma_s(0)=-\}.\]

Note that $\p_p(\sigma_t(0)=-)\ge \p_p(F'|F)\p_p(F)$, and that $\p_p(F)\ge e^{-1}$, so by the strong Markov property it follows that
\[
\p_p(F')\le e\p_p[\sigma_t(0)=-].\]

Now, by Theorem \ref{stronger} and union bound,
for $p>p_0$,

\begin{align*}
 \p_p[\exists s\ge t:\sigma_s(0)=-] & \le \sum_{k=0}^\infty \p_p[\exists s\in[t+k,t+k+1):\sigma_s(0)=-]\\
 & \le \sum_{k=0}^\infty e
 \exp(-(t+k+1)^{\gamma})
  \le e^{-t^{\gamma/2}},
\end{align*}
if $t$ is large enough.
Thus, if 
$F_k:=\{\sigma_s(0)=+,\ \forall s\ge k\}$, and $k_0$ is large

\[\sum_{k\ge k_0}\p_p(F_k^c)\le \sum_{k\ge k_0}\exp(-k^{\gamma/2}) <+\infty.\]

Thus, by the Borel-Cantelli Lemma 
\[\p_p[0\text{ fixates at }+] =\p_p \left[\bigcup_{i\ge 1}\bigcap_{k\ge i}F_k\right]=1.\]

Hence, $p_c^{\textup{vot}}(\Zdd,\UU)\le p_0<1$ and we are finished. 
\end{proof}

At this point, it only remains to show Theorem \ref{stronger}, and the rest of this paper is devoted to its full proof.
To help to understand the overall idea of such proof, we now provide a sketch.
\begin{proof}[Proof Sketch of Theorem \ref{stronger}]
As that proof in \cite{FSS02}, we use a multi-scale analysis; this consists of observing the process in some large boxes $B_k$ at some times $T_k$ which increase rapidly with $k$, and tiling $\Zdd$ with disjoint copies of $B_k$ in the obvious way.
This is done by induction on $k$; $T_0=0$ and suppose we are viewing the evolution inside
the interval $[T_{k-1},T_k)$. In $B_k$ we couple the
process with a {\it block-dynamics} which favors the spins in state $-$ (the $-$ team), in the sense that, when there is
some $-$ in $B_k$ at time $T_k$ in the original process then it
is also true for the block-dynamics.

Inside $B_k$ we allow the $-$ team to `infect' the
$+$ team via their own bootstrap process (meaning that just spins in state $+$
are allowed to flip). We prove that by time $T_k$, every droplet $D\subset B_k$ full
of $-$s has `relatively big' size with small probability. In other words, such droplets satisfy $|D|\ll|B_k|$ with high probability.

Then, we prove that before such droplets $D$ could be created, the $+$ team inside
$B_k$ will typically eliminate it. Moreover, we have to show that the probability that the $-$ team
could receive any help from outside of $B_k$ is also small.

The inductive step goes as follows: at time $T_k$, if there is some $-$ in $B_k$, we declare $B_k$ to be a $-$, otherwise declare
$B_k$ to be $+$, and now, we observe the evolution in a new time interval $[T_k,T_{k+1})$.
The next step, is to consider a larger box $B_{k+1}$ consisting of several copies of $B_k$
that we have declared to be either $-$ or $+$, and we start over again.
By induction on $k$, we will show that if $q:=1-p$ is very close to 0, Theorem \ref{stronger}
holds for all times of the form $t=T_k$. Finally, by using one more coupling trick,
we will be able to extend the statement for all $t\ge 0$.
\end{proof}

\subsection{Bootstrap percolation families}\label{bpf}
First, we review a large class of $d$-dimensional monotone cellular automata, which were
recently introduced by Bollob\'as, Smith and Uzzell \cite{BSU15}, and then focus on dimension two.

Let $\U$ 
be an arbitrary finite family of finite
subsets of $\Zd\setminus \{0\}$. We call $\U$ the {\it update family}, 
each $X\in\U$ an {\it update rule}, and the process itself $\U${\it-bootstrap percolation}.
Now given
a set $A\subset\Zd$ of initially {\it infected} sites, set $A_0=A$, and define for each $t\ge 0$,
\[A_{t+1}=A_t\cup\{x\in\Zd: x+X\subset A_t \text{ for some }X\in\U\}.\]
Thus, a site $x$ becomes infected at time $t+1$ if the translate by $x$ of one of the sets
of the update family is already entirely infected at time $t$, and infected sites remain
infected forever. The set of eventually infected sites is the {\it closure} of $A$, denoted by
$[A]=\bigcup_{t\ge 0}A_t$.

Set $d=2$. Recall that for each $u\in S^1$, we denote $\dhp:=\{x\in\Zdd:\langle x,u\rangle
<0\}$.
We say that $u$ is a {\it stable direction} if $[\dhp]=\dhp$ and we denote by
$\Ss=\Ss(\U)\subset S^1$ the collection of stable directions.
Observe that this definition of $\Ss$ coincides with the one given in Definition \ref{crit}.
The following classification of two-dimensional update families was proposed by Bollob\'as, Smith
and Uzzell \cite{BSU15}.

An update family $\U$ is:
\begin{itemize}
 \item {\it supercritical} if there exists an open semicircle in $S^1$ that is disjoint from $\Ss$;
 \item {\it critical} if there exists a semicircle in $S^1$ that has finite intersection with $\Ss$,
and if every open semicircle in $S^1$ has non-empty intersection with $\Ss$;
 \item {\it subcritical} otherwise. 
 \end{itemize}
 
 The justification for this trichotomy is provided by the next result. Suppose
 we perform the bootstrap percolation process on $\Zdd_n$ instead of $\Zdd$,
$A\subset\Zdd_n$ is $p$-random, and consider the critical probability
 \[p_c(\Zdd_n,\U):=\inf\{p:\p_p([A] = \Zdd_n)\ge 1/2\}.\]

Bollob\'as, Smith and Uzzell \cite{BSU15} proved that the critical
probabilities of supercritical families are polynomial, while those of critical families
are polylogarithmic. Later, Balister, Bollob\'as, Przykucki and Smith \cite{BBPS16} proved that
the critical probabilities of subcritical models are bounded away from zero.
We summarize those results in the following.
\begin{theor}[2-dimensional classification]
 Let $\U$ be a 2-dimensional update family
 \begin{enumerate}
  \item If $\U$ is supercritical then $p_c(\Zdd_n,\U)= n^{-\Theta(1)}$;
  \item If $\U$ is critical then $p_c(\Zdd_n,\U)= (\log n)^{-\Theta(1)}$;
  \item If $\U$ is subcritical then $\liminf p_c(\Zdd_n,\U)>0$.
 \end{enumerate}
\end{theor}

\begin{rema} In the $\U$-voter dynamics, fixation at $+$ should not occur for families which are not critical. For instance,
\begin{itemize}
 \item The families $\U(V)$ introduced in Section \ref{Uvoter} are supercritical (for $V\subset\Zdd$) and do not fixate (see, for example, \cite{CG83}).
 \item The family $\mathcal N_{3}^2$ 
 is subcritical and we do not expect it to fixate at $+$, because
 any translate of $\{1,2\}^2$ that is entirely $-$ at time $t=0$ will remain in state $-$ forever.
 It could be the case that some vertices will fixate at $+$ and others at $-$.
\end{itemize}
\end{rema}
\subsubsection*{Some standard tools}
Let us fix a voter-eroding critical family $\U$. 
We will refer to its associated $\mathcal T$-droplets simply as {\it droplets}.
Now, we introduce an algorithm whose importance is to provide two key lemmas concerning droplets: an ``Aizenman-Lebowitz lemma",
which says that a covered droplet contains covered droplets of all intermediate
sizes, and an extremal lemma, which says that a covered droplet contains a linear
proportion of initially infected sites.

\begin{defi}[Covering algorithm]\label{covalgo}
Suppose $n$ is large and $A\subset\Zdd_n$.
The first step is to choose a sufficiently large constant $\kappa$, fix a droplet $\hat D$ of diameter roughly
$\kappa$, and place a copy of $\hat D$ (arbitrarily) on each element of $A$. Now, at each
step, if two droplets in the current collection are within distance $\kappa$ of
one another, then remove them from the collection, and replace them by
the smallest droplet containing both. This process stops in at most $|A|$ steps with some finite collection
of droplets, say $\{D_1,\dots,D_z\}$.
 \end{defi}

If a droplet occurs at some point in the covering algorithm, then we say that it is {\it covered} by $A$. If $\kappa$ is chosen sufficiently large,
then one can prove that the final collection of droplets covers $[A]_\U$ (see \cite{BSU15} for details).
Now we are ready to state the 2 key lemmas which will help us to control the expanding of the process, their proof can be found in \cite{BSU15}.

Let us write diam$(D)$ for the diameter
of a droplet $D$.

\begin{lemma}[Aizenman-Lebowitz lemma]\label{AL}
 Let $D$ be a covered droplet. Then for every $1\le k\le \textup{diam}(D)$, there is a covered droplet
 $D'\subset D$ such that $k\le \textup{diam}(D')\le 3k$.
\end{lemma}

\begin{lemma}[Extremal lemma]\label{EL}
 There exists a constant $\varepsilon>0$ such that for every covered droplet $D$,
 $|D\cap A|\ge \varepsilon\cdot \textup{diam}(D)$.
\end{lemma}

\section{The one-dimensional approach}\label{Section1da}
Let us fix a critical two-dimensional family $\U$, and let $\Ss$ be its stable set. In this section we prove that if we can find $y\in\Ss$ satisfying certain feasible properties, then we can show that $\U$ is voter-eroding, by using a 1-dimensional argument.
We will consider a particular restricted evolution of the dynamics in dimension 1 by freezing everything except a finite
segment orthogonal to that stable direction, then prove that such a segment can be eroded in polynomial time and show how things can be deduced
from this 1-dimensional setting.

\subsection{A fair stable direction}\label{choosey}
Fix a rational direction $y\in\Ss$ (i.e. $y=(y_1,y_2)$ satisfies either $y_2/y_1$ is rational or $y_1=0$). For each $L\in\mathbb N$, we let $Y=Y(y,L)$ 
be any fixed segment consisting of $L$ consecutive vertices in the discrete line $l_y:=\{x\in\Zdd:\langle x,y\rangle=0\}$.

Suppose we freeze each vertex in $\mathbb H_y$ in state $-$ and each vertex outside $\mathbb H_y\cup Y$ 
in state $+$, and at time $t=0$ each vertex in $ Y$ has state $-$, then we let the dynamics evolve only on
$ Y$ (see Figure \ref{pres}). 
\begin{figure}[ht]
  \centering
    \includegraphics[scale=.4]{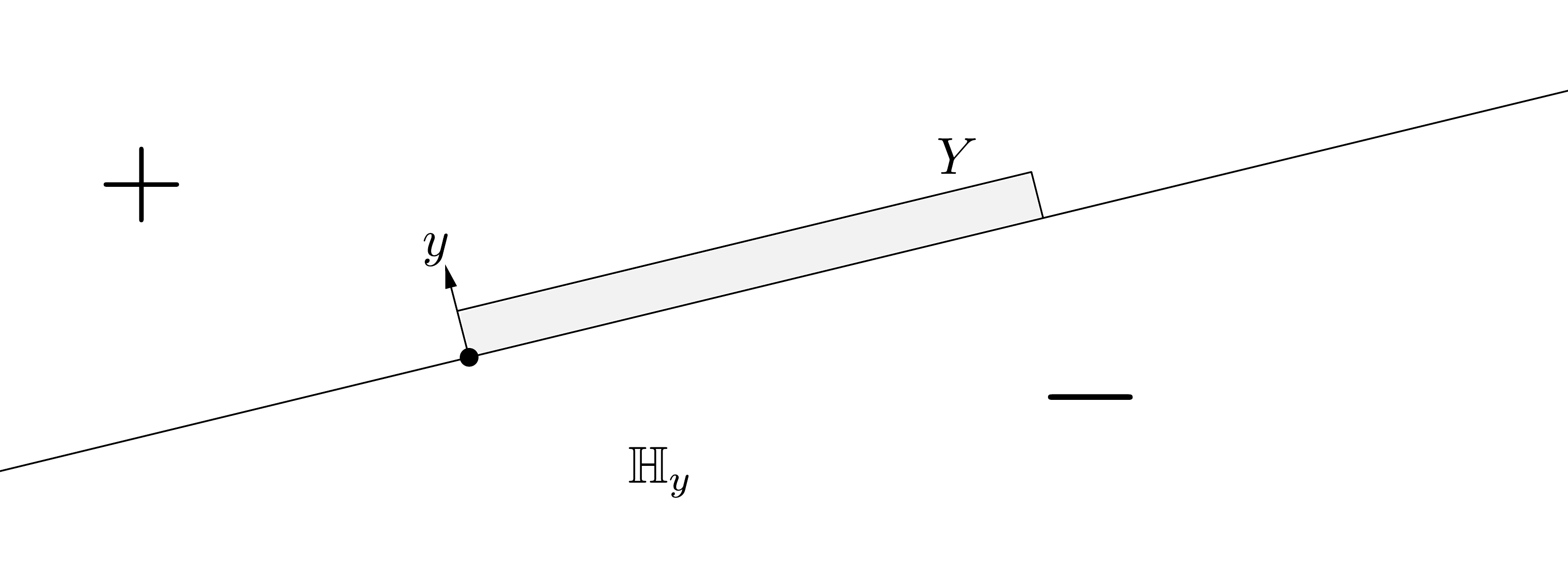}
  \caption{$\mathbb H_y$ entirely $-$ and $\Zdd\setminus (\mathbb H_y\cup Y)$ entirely $+$}
  \label{pres}
\end{figure}

Given a configuration $\eta\in\{+,-\}^Y$ denote $\eta^+$ (resp. $\eta^-$) the set of vertices in $\eta$ having $+$ (resp. $-$) state.
\begin{defi}\label{etat}
\begin{enumerate}
 \item Fix $y\in\Ss$, $L\in\mathbb N$, and consider $Y=Y(y,L)$. For every $t\ge 0$ we let $\eta_t$ denote the configuration in $\{+,-\}^Y$ at time $t$ in these restricted 1-dimensional dynamics.
 \item  Say that $y\in \Ss$ is a {\it fair direction} if it is rational, and in the 1-dimensional dynamics, for each $L\in\mathbb N$ and $t\ge 0$ the following holds
 \begin{equation}\label{presym}
 \sum\limits_{v\in\eta^+_t}r_v(\eta_t)\le \sum_{u\in\eta^-_t}r_u(\eta_t).
 \end{equation}
\end{enumerate}
\end{defi}
Denote by $[-]$ (resp. $[+]$) the configuration in $\{+,-\}^Y$ where all vertices are in state $-$ (resp. +),
and observe that Condition (\ref{presym}) is trivial for $t=0$ since $\eta_0=$ $[-]$ and this gives LHS in (\ref{presym}) equals 0. 
Moreover, when $t$ is large  $\eta_t=$ $[+]$ (the segment fixates at $[+]$) and $r_v([+])=0$ for all $v$ (since $y\in\Ss$), hence LHS $=0$ too.

Note that Condition (\ref{presym}) is implied by the stronger condition
\begin{equation}\label{presymstronger}
\sum\limits_{v\in\eta^+}r_v(\eta)\le \sum_{u\in\eta^-}r_u(\eta), \textup{ for all }\eta\in\{+,-\}^Y,
\end{equation}
which does not depend on the trajectory of the 1-dimensional dynamics $\eta_t$. However, in general we do not know whether they are equivalent or not.
 
Our aim now is to show that the existence of a fair direction is a sufficient condition for a family to be voter-eroding.
Given a fair direction $y$, we are interested in the {\it segment erosion time}
\begin{equation}
 \tau=\tau(y,L):=\inf\{t:\eta_t=[+]\}.
\end{equation}

Here is the core of the 1-dimensional approach.
\begin{prop}\label{poltime}
If there is a fair direction $y$, then there is a constant $c>3$ such that for $L$ large enough we have $\p[\tau > L^{c-1}]\le e^{-1}$.
\end{prop}
We move the proof of this proposition to the next section. 
This result takes account of the constant $c>1$ in Definition \ref{chuloenbola}, so, it is left to choose a good subset of the stable set; that is the content of the next lemma.
Let us denote the convex hull of a set $S$ by Hull($S$).
\begin{lemma}\label{existeS4}
Given $y\in\Ss$, there exists a finite set $\Ss_4\subset\Ss$ 
such that $y\in\Ss_4$ and $0\in\textup{Hull}(\Ss_4)$.
\end{lemma}

Before proving this lemma, we introduce some useful information about the structure of $\Ss$. Write $[u,v]$ for the closed interval of directions between $u$ and $v$
(also $(u,v)$ for the open interval). Say $[u,v]$ is {\it rational} if both $u$ and $v$ are rational
directions. Our choice of $\Ss_4$ will depend on the following lemma.
\begin{lemma}\label{Sclosed}
 The stable set $\Ss$ is a finite union of rational closed intervals of $S^1$.
\end{lemma}
\begin{proof}
See \cite{BSU15}.
\end{proof}

With this tool, now we prove that in fact the set $\Ss_4$ in Lemma \ref{existeS4} can be chosen of size 3 or 4 (thus, justifying the subindex 4). 
\begin{proof}[Proof of Lemma \ref{existeS4}] Let $y\in\Ss_4$ by definition. 
If $-y\in \Ss$ since $\UU$ is critical we can choose $x\in(y,-y)\cap\Ss$,
$z\in(-y,y)\cap\Ss$ and set $\Ss_4=\{x,y,-y,z\}$.

If $-y\notin \Ss$, then take $x\in\Ss$ in the open semicircle opposite to $ y$,
we can suppose wlog that $x\in( y,- y)$.
\begin{figure}[ht]
  \centering
    \includegraphics[scale=.7]{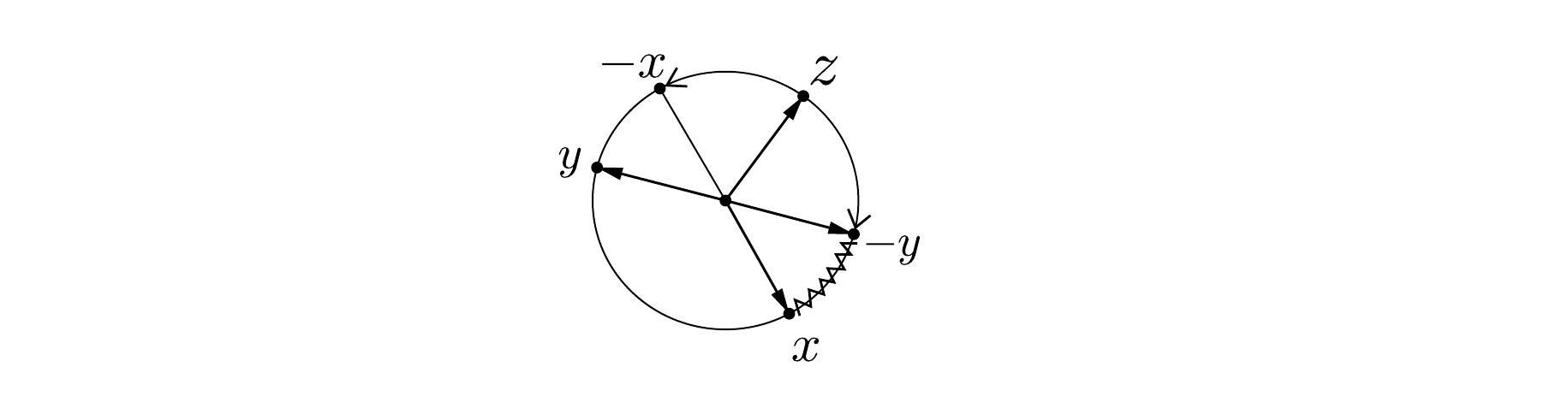}
  \caption{3 or 4 stable directions}
   \label{3-4stabledir}
\end{figure}

Moreover, since $\Ss$ is closed by Lemma \ref{Sclosed}, we can choose $x$ such that $\Ss\cap[x,- y]=\{x\}$ (see Figure \ref{3-4stabledir}).
Then select $z\in\Ss\cap(x,-x)$ and observe that in fact $z\in\Ss\cap(- y,-x)$,
so define $\Ss_4=\{x,y,z\}$. In both cases, $0\in \text{Hull}(\Ss_4)$.
\end{proof}

By combining the previous results, we can prove that every family with fair directions is voter-eroding.
\begin{prop}\label{erosion1}  
If there is a fair direction, then there exist a constant $c>1$, and a set $\Ss_4\subset\Ss$ such that for every $(\Ss_4,L)$-droplet $D$ and every $t\ge 0$,
 \begin{equation}
  \p_p[T(D)> tL^{c-1}]\le Le^{-t},
 \end{equation}  
 when $L$ is large enough. In particular, $\UU$ is $(c+\varepsilon,\Ss_4)$-eroding, for any $\varepsilon>0$.
\end{prop}
\begin{proof}
To fix ideas, we can assume that $y=(0,1)$ is a fair direction. Consider the set $\Ss_4$ containing $y$, given by Lemma \ref{existeS4} and any $(\Ss_4,L)$-droplet $D$. 
Note that $D$ is finite because $0\in \textup{Hull}(\Ss_4)$, 
hence $T(D)$ is well defined.

We couple the dynamics with the following one:
We first allow to flip just the vertices in the first (top) line of the droplet, then, when they are all in state $+$, we allow
to flip just the vertices in the second line, and so on until we arrive at the bottom line.
This coupled dynamic dominates the original one by monotonicity, and since the height of the droplet is at most $L$, then it is enough to show that for all $t\ge 0$,
\begin{equation}\label{cgama}
  \p_p[T_{\textup{top}}> tL^{c-1}]\le e^{-t},
 \end{equation}
 where $T_{\textup{top}}$ is the time to erode the top line in the coupled dynamics; we finish the proof by applying the union bound over all rows of $D$. 
 
 Moreover, we can assume that this top line has
 $L$ vertices, 
 since all lines have at most
 $L$ vertices and having less vertices only helps to erode faster. 
 To this end, let us consider the 1-dimensional process
 $\eta_t$ given in Definition \ref{etat}; because of the boundary conditions, it follows that $T_{\textup{top}}\le\tau$ in distribution, thus, by Proposition \ref{poltime},
 \[\p_p[T_{\textup{top}}> L^{c-1}] \le \p[\tau > L^{c-1}]\le e^{-1}.\]
 Finally, by the Markov property it follows \eqref{cgama}.
\end{proof}

\subsection{A martingale argument}
In this section we prove Proposition \ref{poltime}. To do so, we use Markov's inequality $\p[\tau>s]\le \e[\tau]/s.$
The first step is to show that,
if we can find a function on $\{+,-\}^Y$ providing a bias in favor to the vertices in state $+$ in the dynamics $\eta_t$ (see Definition \ref{etat}), then we can bound $\e[\tau]$ in terms of $f$ and the extreme configurations $[-]$ and $[+]$.

Since $Y=Y(y,L)$ has $L$ vertices can identify it
with the initial segment $[L]$, so we can write the generator for the 1-dimensional process $\eta_t$ as
\[\mathbb V f(\eta)=\sum_{v=1}^L\frac{r_v(\eta)}{m}[f(\eta^v)-f(\eta)].\]
\begin{lemma}\label{pairfeps}
Suppose there exists a function $f:\{+,-\}^L\to\R$ such that $\mathbb V f(\eta_t)\le -1$ for all $t<\tau$, then
\end{lemma}
\vspace{-.5cm}
\begin{equation}\label{+-polinomial}
\e[\tau]\le f([-])-f([+]).
\end{equation}
\begin{proof}
Consider the martingale $M_t=f(\eta_t)-\int_0^t\mathbb V f(\eta_s)\,ds$. By optional stopping we have
\begin{align*}
 f([-])&=\e[M_0] =\e[M_{\tau}]
=\e\left[f(\eta_{\tau})-\int_0^{\tau}\mathbb V f(\eta_s)\,ds\right]\\
 &\ge \e[f(\eta_{\tau})]+\e\left[\int_0^{\tau} 1\,ds\right]
 = f([+])+\e[\tau],
\end{align*}
and the result follows.
\end{proof}

In order to apply this result, the next step is to show that if $y$ is a fair direction then we can define an explicit function $f$ such that the variation $f([-])-f([+])$ is polynomial in $L$.
\begin{prop}\label{gluing}
If $y$ is a fair direction then there exists a function $f$ satisfying the
 hypothesis in Lemma \ref{pairfeps} such that RHS in (\ref{+-polinomial}) is $O(L^2)$.
\end{prop}
\begin{proof}
Set $h_0=0$ and for $k=0,1,\dots,L-1$ consider the sequence
\[h_{k+1}=h_k+(L-k)m.\] 
Then, define the function $f$  
as follows: given $\eta\in\{+,-\}^L$ with $k=k(\eta)$ vertices in state $-$ we set $f(\eta)=h_k$. Observe that

\begin{align*}
f([-])-f([+])
=\sum_{k=0}^{L-1}[h_{k+1}-h_k]
=\sum_{k=0}^{L-1}(L-k)m 
= O(L^2).
\end{align*}
Moreover, given $\eta=\eta_t$ with $k=k(\eta)\ge 1$ vertices in state $-$ we have
\begin{align*}
m[\mathbb V f](\eta) &=\sum_{v=1}^Lr_v(\eta)[f(\eta^v)-f(\eta)]\\
& =-\sum_{v\in\eta^-}r_v(\eta)[f(\eta)-f(\eta^v)]+ \sum_{v\in\eta^+}r_v(\eta)[f(\eta^v)-f(\eta)]\\
& =-\sum_{v\in\eta^-}r_v(\eta)[h_k-h_{k-1}]+ \sum_{v\in\eta^+}r_v(\eta)[h_{k+1}-h_k]\\
& =-\sum_{v\in\eta^-}r_v(\eta)[m(L-(k-1))]+ \sum_{v\in\eta^+}r_v(\eta)[m(L-k)]\\
&\le -\sum_{v\in\eta^-}r_v(\eta)m(L-(k-1))+ \sum_{v\in\eta^-}r_v(\eta)m(L-k)\\
&= -\left(\sum_{v\in\eta^-}r_v(\eta)\right)m
\le -m.
\end{align*}
So $\mathbb V f(\eta)\le -1$ for all $\eta\ne[+]$ and we are done.
\end{proof}

Finally, we are ready to conclude.
\begin{proof}[Proof of Proposition \ref{poltime}]
 If there is a fair direction $y$ then apply Proposition \ref{gluing} and then Lemma \ref{pairfeps} to get
 \[\e[\tau]=O(L^2)\le e^{-1} L^{c-1},\] for any constant $c>3$, and for $L$ large enough, so by applying Markov's inequality we are finished.
\end{proof}

\subsection{Examples}\label{Exampless}
It is easy to show that $\mathcal N_2^2$ and Duarte model (discussed in the Introduction) verify Condition (\ref{presym}). In fact,
$y=e_2$ is a fair direction for both families, and all configurations $\eta_t$ in the 1-dimensional dynamics are of the form
\[\eta_t=[+,\cdots,+,-,\cdots,-,+,\cdots,+],\]
this means, a block of $+$s followed by a block of $-$s followed by another block of $+$s; the right-most block of $+$s being empty for the Duarte family, and either block of $+$s (left or right) could be empty for $\mathcal N_2^2$.

To fix ideas, let us consider the family $\mathcal N_2^2$ and assume that only the right-most block is empty, thus
\[\sum\limits_{v\in\eta^+_t}r_v(\eta_t)= 1,
\textup{ and } 2
\le \sum_{u\in\eta^-_t}r_u(\eta_t).\]
The other configurations and Duarte family can be checked in a similar way. Therefore, inequality (\ref{presym}) holds. 

Now, we give a criterion (that only depends on the family $\UU$) which allows us to check if there exists a fair direction just by drawing the rules of $\U$ in $\Zdd$.
Given $y\in \Ss$, consider the line $l_y:=\{x\in\Zdd:\langle x,y\rangle=0\}$. 
\begin{prop}
If there exists a rational direction $y\in\Ss$ satisfying:
 \begin{enumerate}\label{lemadraw}
  \item[(a)] each $X\in\UU$ has either, at most 1 vertex in $l_{y}$, or at least 1 vertex in $\mathbb H_{-y}$, and
  \item[(b)] there exists an injective function $g:\U\to \U$ such that, $X\in\UU$ and $X\subset\mathbb H_{y}\cup\{x\}$ for some $x\in l_y$ implies $g(X)\subset\mathbb H_{-y}\cup\{-x\}$,
 \end{enumerate}
 then $y$ is a fair direction.
\end{prop}


\begin{proof}
 In fact, fix $L\in\mathbb N$, and consider any configuration in $\eta\in\{+,-\}^{Y}$.
 We will show that
\begin{equation}\label{presym2}
 \sum\limits_{v\in\eta^+}r_v(\eta)\le \sum_{u\in\eta^-}r_u(\eta).
 \end{equation}
 
 First of all, note that rules $X$ having at least 1 vertex in $\mathbb H_{-y}$ do not contribute to the LHS of (\ref{presym2}).
 If $(v,X)\in\eta^+\times\UU$ is counted in LHS of (\ref{presym2}),
 this is because $v\in\eta^+$ and $X$ disagrees with $v$, hence $v+X$ is entirely $-$. Moreover, since
 $y\in\Ss$, the set $v+X$ must have a vertex $u=u(v)\in\eta^-$, and only 1 by {\it (a)}, so
 $v+X\subset\mathbb H_{y}\cup\{u\}$ or $X\subset\mathbb H_{y}\cup\{u-v\}$.
 
 Now, by {\it (b)}, $g(X)\subset\mathbb H_{-y}\cup\{v-u\}$, or $u+g(X)\subset\mathbb H_{-y}\cup\{v\}$, thus, $u+g(X)$ in entirely $+$, 
 meaning that  $g(X)$ disagrees with $u$ and $(u,g(X))\in\eta^-\times\UU$ is counted in RHS of (\ref{presym2}).
 Thus, in order to prove that inequality (\ref{presym2}) holds, it is enough to check that for each pair $(v,X)$ that contributes 1 in LHS we can find a contribution of 1 in RHS in a one to one way, in other words, that the map $g':(v,X)\mapsto (u,g(X))$ is an injection. 
 
 In fact, if $g'(v_1,X_1)=g'(v_2,X_2)$, then $u(v_1)=u(v_2)$, and $X_1=X_2$ (since $g$ is injective). So, $X_1\subset\mathbb H_{y}\cup\{u(v_1)-v_1\}$ and $X_1\subset\mathbb H_{y}\cup\{u(v_1)-v_2\}$. Finally, by {\it (a)} we conclude that $v_1=v_2$ and $g'$ is injective.
\end{proof}

Observe that Condition {\it (b)} (but not {\it (a)} in general) holds for symmetric families (i.e., when $X\in\UU$ implies $-X\in\UU$),
because $g(X)=-X$ works.
By using this proposition, we are able to construct a large class of families having $e_2$ (wlog) as a fair direction, as follows.
\begin{exam}
Fix a two-dimensional family $\UU'$  satisfying
\begin{enumerate}
 \item[(I)] $\forall X\in\UU', X\cap\mathbb H_{-e_2}\ne\emptyset$ and $0\in$\,Hull$(X)$,
 \end{enumerate}
and let $\mathcal V,\mathcal W$ be 1-dimensional families such that either
 \begin{enumerate}
 \item[(II)] $0<\nu_- \le w_+$ and
 $w_-\le \nu_+$, or
 \item[(II')]
 $\mathcal V=\emptyset$ and $w_+w_->0$,
\end{enumerate}
where $\nu_*$ (resp. $w_*$) denotes the number of rules of $\mathcal V$ (resp. $\mathcal W$) entirely contained in $\mathbb Z_*$ ($\mathbb Z_+=\mathbb N$ and $\mathbb Z_-=-\mathbb N$).
Then, for every $i\in\mathbb Z_+$, the following {induced} two-dimensional family is voter-eroding:
\[\UU_i(\UU',\mathcal V,\mathcal W):=\UU'
\cup\bigcup_{R\in\mathcal V}\{X(i,R)\}\cup\bigcup_{R\in\mathcal W}\{X(-i,R)\},\]
where $X(\pm i,R)$ denotes the rule $\{\pm ie_1\}\cup\{re_2:r\in R\}$.

In fact, $\UU_i(\UU',\mathcal V,\mathcal W)$ satisfies {\it (a)}, since either $ie_1$, or $-ie_1$, is the only vertex in $l_{e_2}$ and in some rule $X(\pm ie_1, R)$ at the same time.
Condition {\it (b)} is also satisfied, since we can take $g(X)=X$ for $X\in \U'$, and it is easy to see that conditions (II)/(II') ensure that we can make an one to one assignment to those rules $X\notin \U'$. 

Specific examples of families satisfying (I) and (II)/(II') are
\[\UU_1(\{\{1,-1\}\},\ \{\{1\},\{-1\}\},\ \{\{1\},\{-1\}\})= 
\mathcal N_2^2\setminus\{\{e_1,-e_1\}\},\]
and
\[\UU_1(\{\{1,-1\}\},\ \emptyset,\ \{\{1\},\{-1\}\})=
\{\{e_2,-e_2\},\{-e_1,e_2\},\{-e_1,-e_2\}\},\]
which is the same as the Duarte model.

The second example mentioned in the introduction (see Example \ref{easyexamples}) is very similar to $\UU_2(\UU',\mathcal V,\mathcal W)$, with $\UU'=\{(1,1), -e_2\}$,  $\mathcal V=\{\{-3,-2,-1\},\{3,5\}\}$ and $\mathcal W=\{\{-4\},\{1\}\}$.
Indeed, note that the above proposition also implies that given $i_1,\dots,i_k\in\mathbb Z_+$, the family
\[\UU_{i_1}(\UU'_1,\mathcal V_1,\mathcal W_1) 
\cup\cdots\cup
\UU_{i_k}(\UU'_k,\mathcal V_k,\mathcal W_k)\]
is voter-eroding, as soon as $\UU'_j,\mathcal V_j,\mathcal W_j$ satisfy (I) and (II)/(II') for each $j\le k$.
\end{exam}

On the other hand, we are free to construct plenty of examples of different nature which satisfy conditions {\it (a)} and {\it (b)}, by using the following trivial observation.
\begin{rema}[Adding rules]\label{addrules}
 Infinitely many families with a fair direction $y$ can be constructed from a single family $\UU$, just by properly adding new rules $X\subset\mathbb H_{-y}$.
 Moreover, 
 if $\UU$ critical, then the new families can be chosen critical as well, we just need to add new rules carefully, without modifying (too much) the stable set (see Remark \ref{addrules2}). 
\end{rema}

\subsubsection*{No fair directions}
For a concrete example of critical families without fair directions consider the collection of all subsets of size 3 of
\[\{\pm e_1,\pm e_2,\pm 2e_1,\pm 2e_2\},\] 
call it $\UU_{3,8}$. It is easy to check that
$\Ss(\UU_{3,8})=\{\pm e_1,\pm e_2\}$.
To check that no direction in $\Ss(\UU_{3,8})$ is fair, by symmetry it is enough to consider $y=e_2$; in the 1-dimensional setting,
calculations show that configurations $\eta$ of the form
\[\eta=[-,+,+,+,+,-,+,+,+,+,-,+,+,+,+,-,
\cdots,
+,+,+,+,-],\]
which alternate four $+$s and one $-$, do not satisfy (\ref{presym}).
However, simulations indicate that  
we can erode the segment perpendicular to $y$ in time $O(L^{2.2})$, which would mean that $\Ss(\UU_{3,8})$ is voter-eroding. 

 Another such family, which is special since its droplets are triangular, is
\[\UU_\rhd=\{\{(-1,1),(-1,-1)\}, \{(0,1),(1,1)\}, \{(0,-1),(1,-1)\}\},\]
with $\Ss(\UU_\rhd)=\left\{-e_1, \frac{1}{\sqrt 2}(1,1), \frac{1}{\sqrt 2}(1,-1)\right\}$ (see Figure \ref{nonpres}).
\begin{figure}[ht]
  \centering
    \includegraphics[scale=.85]{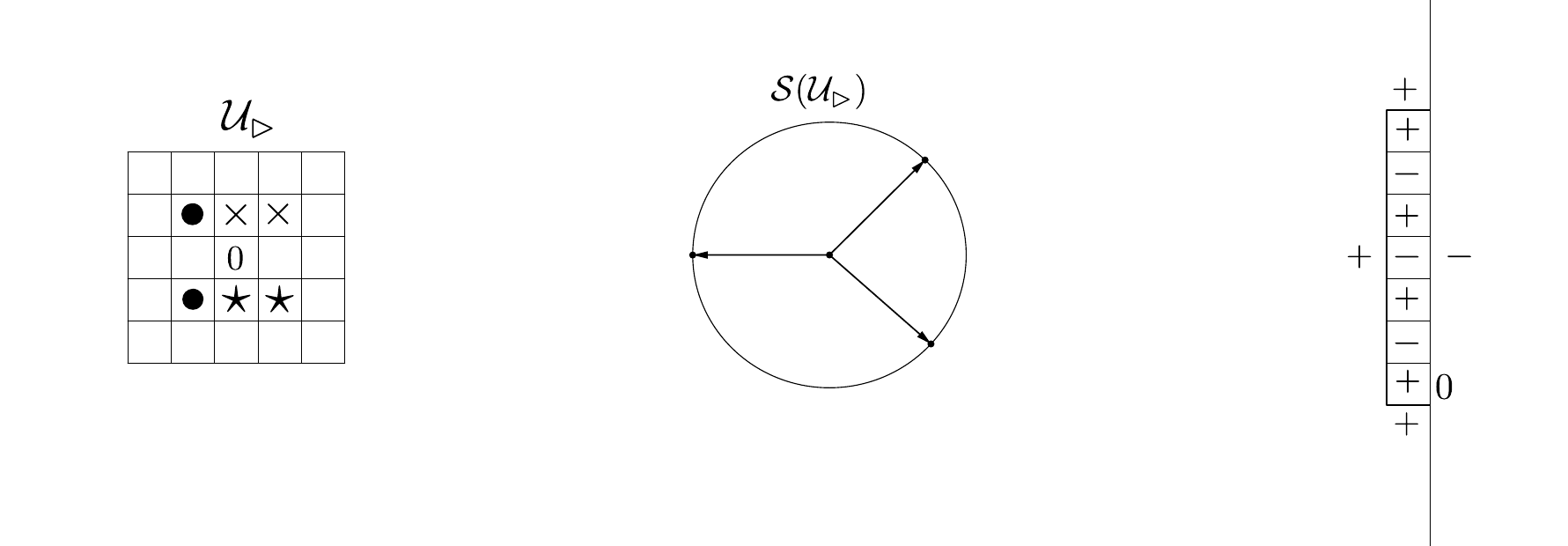}
  \caption{A family without fair directions, its stable set and a failing configuration}
  \label{nonpres}
\end{figure}

This family only has 3 candidates to be $y$ and all of them fail Condition (\ref{presym}).
In fact, say we choose $y=-e_1$ and take $L=2n+1$ for some fixed $n$. The configuration $\eta$ given by \[\eta^+=\{(0,2k):k=0,...,n\},\]
yields $\sum_{u\in\eta^-}r_u(\eta)=n$, while $\sum_{v\in\eta^+}r_v(\eta)=2n$. An analogous situation happens for the other 2 candidates.
On the other hand, simulations suggest that
in fact we can erode the segment perpendicular to $y=-e_1$  in time $O(L^{4/5})$.
  
As a last example, we illustrate how to construct a voter-eroding family from $\UU_\rhd$, in the sense of
Remark \ref{addrules}.
\begin{rema}\label{addrules2}
 The direction $y=-e_1$ is fair for the family \[\UU_\rhd\cup\{\{(-1,2),(-1,-1)\}\},\] this is just because we added a rule in 
 the + side to
 compensate inequality (\ref{presym2}) 
 since the it has the same stable set as $\UU_\rhd$ then it is critical and voter-eroding so our main theorem holds
 for this new family.

 In general, if we consider any rule $X_0\subset\mathbb H_{e_1}$, $X_0\ne \{(-1,1), (-1,-1)\}$ such that there exist
 $x,x'\in X_0$ with $x_2\ge -x_1$ and $x_2'\le x_1'$, we can check that $-e_1$ is a fair direction for the family $\UU_\rhd\cup\{X_0\}$,  and the latter has the same stable set as $\UU_\rhd$. Of course,
 we can construct infinitely many such families in this way. 
\end{rema}

\section{The process inside rectangles}\label{SectioninsideR}
We move to the proof of Theorem \ref{stronger}, let us fix a $(c,\mathcal{T})$-eroding critical family $\UU$.
The strategy starts by constructing a rapidly decreasing sequence $\{q_k\}_k$ with $q_0=q=1-p$, 
and study the process in big space and time scales. More precisely, set
$l_0=1, t_0=0$, and define for $k\ge 1$ the sequences
\begin{equation}
q_k:=\exp(-a/q_{k-1}),    
\end{equation}
\begin{equation}
 l_k:=\left\lfloor\left(\frac{1}{q_{k-1}}\right)^{3c}\right\rfloor,
\ \ \ \ t_k:=\left(\frac{1}{q_{k-1}}\right)^{2c},
\end{equation}
\begin{equation}
L_k:=\prod_{i=0}^kl_i,\ \ \ \ \ \ \ \ \
T_k:=\sum_{i=0}^kt_i.
\end{equation}

Here, we recall an upper bound
for $L_k$ in terms of $q_k$, which was computed in \cite{FSS02}.
\begin{lemma}
If $q$ is small enough, then for every $k\in\mathbb N$ we have
 \begin{equation}\label{delta}
L_{k}\le 1/q_k.
 \end{equation}
\end{lemma}
\begin{proof}
By induction on $k$. For further details, see equation (4.8) in \cite{FSS02}.
\end{proof}
Now, consider the squares
\begin{equation}
    B_k:=[L_k]^2.
\end{equation}
At time $T_k$ we tile $\Zdd$ into copies of $B_k$ in the obvious way,
then, couple the $\UU$-voter dynamics with a {\it block-dynamics} (which is more `generous' to the spins in state
$-$), defined as follows: for every $k\ge 0$,
\begin{itemize}
 \item At time $T_k$ every copy of $B_k$ is monochromatic and the $\UU$-voter dynamics afresh;
 until it arrives at time $T_{k+1}$.
 \item As $t$ is close to $T_{k+1}$, if there exists some copy of $B_k$ inside some copy of $B_{k+1}$ which is in
 state $-$, then at time $T_{k+1}$ we declare the state of $B_{k+1}$ to be $-$. Otherwise we declare
 it to be $+$.
\end{itemize}
\begin{defi}
Define $\hat q_k$ as the probability that at time $T_k$ the block $B_k$ is in the state $-$ in the
block-dynamics.
\end{defi}

Note that $\hat q_0= q_0$. The following inequality is the core of the proof.
\begin{prop}\label{lema1}
If $q$ is small enough, then $\hat q_k\le q_k$, for every $k$.
\end{prop}
In order to prove this proposition, we proceed by induction.
Let us assume that it holds for $k$, then we will prove a series of lemmas, and finally deduce that the proposition holds for $k+1$ in Subsection \ref{pruebadepropcore}. 

Let $B_k'$ be the block with the same center as $B_k$, but with sidelength $\frac 53L_k$. 
\begin{defi}
We define $\mathcal{P}$ to be the process  obtained by running the $\UU$-voter dynamics only on the induced graph $\Zdd[B_{k+1}']$, with $+$ boundary conditions.
\end{defi}
Consider the evolution of $\mathcal P$ during the interval $[T_k,T_{k+1})$ and define the event
\begin{equation}
 F_{k+1}:=\{\exists -\in B_{k+1} \text{ as } t 
 \nearrow T_{k+1}\}.
\end{equation}

In the next two subsections we will prove the
\begin{lemma}\label{lema2}
$\p_p(F_{k+1})\le q_{k+1}/2$.
\end{lemma}

\subsection{Bootstrapping the vertices in state $-$}
Set $n=l_{k+1}$ and identify each vertex in $[n]^2$ with a copy of 
$[L_k]^2$. Consider $\UU$-bootstrap percolation on $[n]^2$ by setting the initially infected set $A$ to consist of all vertices in state $-$.
Then run the covering algorithm (see Definition \ref{covalgo}), by declaring the infected sites to be those in state $-$ (thus, spins $+$ become $-$), until we stop with a finite collection of droplets, say $\{D_1,\dots,D_z\}$, each one entirely $-$.
Consider the event
\begin{equation}
 E=\left\{\text{diam}(D_i)\le\varepsilon^{2}n^{1/3c}, \text{ for all }i=1,\dots,z\right\},
\end{equation}
where $\varepsilon>0$ is the constant given by Lemma \ref{EL}.
\begin{lemma}\label{lema3}
$\p_p(E^c)\le {q_{k+1}}/{4}$.
\end{lemma}
\begin{proof}
If some $D_i$ has diameter bigger than $\varepsilon^{2}n^{1/3c}$,
then by Lemma \ref{AL} there exists a covered droplet $D$ with $\varepsilon^{2}n^{1/3c}\le
\text{diam}(D)\le 3\varepsilon^{2}n^{1/3c}$. If $N$ denotes the number of such droplets when $A$ is $q_k$-random, then by Markov's inequality we have
\begin{align*}
 \p_p(E^c) \le
 \e_p[N] &\le \sum_{s=\varepsilon^{2}n^{1/3c}}^{3\varepsilon^{2}n^{1/3c}}n^3{{s^2}\choose{\varepsilon s}}
q_k^{\varepsilon s}
\le \sum_{s=\varepsilon^{2}n^{1/3c}}^{3\varepsilon^{2}n^{1/3c}}n^3\left(\frac{esq_k}{\varepsilon}
\right)^{\varepsilon s}
\\ & \le n^3\sum_{s=\varepsilon^{2}n^{1/3c}}^{3\varepsilon^{2}n^{1/3c}}\left( 3e\varepsilon\right)^{\varepsilon s}
\le C n^3\exp\left(-\varepsilon^{2}n^{1/3c}\right)\\
& \le\frac{\exp\left(-2an^{1/3c}\right)}{4},
\end{align*}
(here we used $n^{1/3c}q_k\le 1$ and picked $a<\varepsilon^2/2$), by Lemma \ref{EL}, since the number of droplets in $\Zdd_n$
with diameter $s$ is $O(n^{2+1/3c})$, and each has area at most $s^2$.
Finally, note that $\exp(-2an^{1/3c})\le q_{k+1}$.
\end{proof}
\begin{lemma}\label{lema3'}
If $a$ is small enough, then \begin{equation}\label{step2}
\p_p(F_{k+1}|E)\le{q_{k+1}}/{4} 
\end{equation}
\end{lemma}
We move the proof of this statement to the next subsection. Now, it is easy to deduce Lemma \ref{lema2}.
\begin{proof}[Proof of Lemma \ref{lema2}]
Note that $\p(F_{k+1})\le\p_p(F_{k+1}|E)+\p_p(E^c)$
and apply Lemmas \ref{lema3} and \ref{lema3'}.
\end{proof}

\subsection{Erosion step}
To get \eqref{step2} we will use the voter-eroding property of $\UU$. 
We need to estimate the probability that
starting at time $T_k$ from a configuration in $B_{k+1}$ where $E$ holds and letting the system evolve with + boundary
conditions, some spin $-$ will be present as $t$ is close to $T_{k+1}$. 
An upper bound is obtained by
starting the evolution at time $T_k$ with $-$ spins at all sites of the droplets $D_1,\dots,D_z$ participating in $E$.
\begin{proof}[Proof of Lemma \ref{lema3'}]
If $E$ occurs, by \eqref{delta}, for small $q$ each droplet has diameter at most 
\[\varepsilon^{2}l_{k+1}^{1/3c}\cdot L_k\le \varepsilon^{2}\frac{1}{q_{k}}\left(\frac{1}{q_{k}}\right)
\le \left(\frac{1}{q_{k}}\right)^{2},\]
hence, each $D_i$ is a $\left(\mathcal T, (1/q_k)^{2}\right)$-droplet.
Since $\UU$ is $(c,\mathcal T)$-eroding, if $q$ is small enough, for each $i=1,\dots,z$, the probability that at time $T_{k+1}=T_k+ (1/q_{k})^{(2)c}$
there is any spin $-$ inside $D_i$ is at most
\[\p_p\left[T(D_i)>\left(\frac{1}{q_{k}}\right)^{2c}\right]
\le
\exp\left(- \left(\frac{1}{q_{k}}\right)^{2}\right).\]

For small $q$ we also have $z\le |B_{k+1}'|\le [(5/3)L_{k+1}]^2\le 1/{q_{k+1}}$, by \eqref{delta} again. Therefore
\[\p_p[F_{k+1}|E]\le \frac{1}{q_{k+1}}\exp\left(- \left(\frac{1}{q_{k}}\right)^{2}\right)\le \frac{q_{k+1}}{4},\]
for $q$ small.  
\end{proof}

\section{Wrapping up}\label{SectionWrapup}
In this last section, we finish the proof of Theorem \ref{stronger}.
\subsection{Control of the outer influence: Proof of Proposition \ref{lema1}}\label{pruebadepropcore}
The strategy now will be the following: we will prove that the probability that by time $T_{k+1}$ the state of every vertex in $B_{k+1}$ in the original dynamics differs from the process $\mathcal{P}$ is small.
We will do this by arguing that the probability of having some $-$ inside $B_{k+1}$ with the help of some vertex outside $B_{k+1}'$ on time $[T_k,T_{k+1})$ is small. Then, by using Lemma \ref{lema2}, we will deduce Proposition \ref{lema1}.
Finally, we will put all the pieces together and deduce Theorem \ref{stronger}.

\begin{defi}
We call a sequence $(x_1,s_1),\dots,(x_r,s_r)$ of vertex-time pairs, where $x_i\in\Zdd$ and $s_i\ge 0$,
a {\it path of clock rings} (and say that such a sequence is a path from $x_1$ to $x_r$ in time
$[s_1,s_r]$) if
\begin{enumerate}
 \item $0<\|x_{i+1}-x_i\|_1\le C$ for each $i\in[r-1]$, where $C=\max\limits_{X\in \UU}\{\|x\|:x\in X\}$.
 \item $s_1<\cdots<s_r$.
 \item The clock of vertex $x_i$ rings at time $s_i$ for each $i=1,\dots,r$.
\end{enumerate}
\end{defi}
The key point now is that
if there does not exist a path of clock-rings from $x_1$ to $x_r$ in time $[s_1,s_r]$, then the state
of vertex $x_r$ at time $s_r$ is independent of the state of vertex $x_1$ at time $s_1$.
\begin{lemma}\label{lema4}
If $F_{k+1}'$ is the event that there exists a path of clock-rings from some vertex outside
$B_{k+1}'$ to some vertex inside $B_{k+1}$ in time $[T_k,T_{k+1}]$, then
\begin{equation}\label{Fk'}
 \p_p(F_{k+1}')\le q_{k+1}/2.
\end{equation}
\end{lemma}
With this estimate, we are ready to finish our induction step.
\begin{proof}[Proof of Proposition \ref{lema1}]
If $F_{k+1}'$ does not occur,
then the state of every vertex in $B_{k+1}$ at time $T_{k+1}$ is the same in the $\UU$-voter dynamics as it is in the process $\mathcal P$, since the boundary conditions cannot affect $B_{k+1}$.
This gives $\hat q_{k+1}\le \p_p(F_{k+1})+ \p_p(F_{k+1}')$, and the result follows from Lemmas \ref{lema2} and \ref{lema4}.
\end{proof}
It only remains to show inequality (\ref{Fk'}).
\begin{proof}[Proof of Lemma \ref{lema4}]
If $F_{k+1}'$ occurs, every such a path have length at least \[r_k:= \left\lfloor\frac{1}{3C}
L_{k+1}\right\rfloor\ge\frac{1}{6C}\left(\frac{1}{q_{k}}\right)^{3c}.\]
By \eqref{delta}, for each $r\in\mathbb N$, the number of paths of length $r$ starting on 
the boundary of $B_{k+1}'$ is at most
\[O\left(L_{k+1}(4C^2)^r\right)\le \frac{1}{q_{k+1}}(4C^2)^r.\] 

Let $P_k(r)$ be the probability that there exist times $T_k\le s_1<\cdots<s_r\le T_{k+1}$ such that $(x_1,s_1),\dots,(x_r,s_r)$ is a path of clock rings.
Observe that $P_k(r)$ does not depend on the choice of the path, since all clocks have the same distribution. 
We have to bound $P_k(r)$.

Set $s_0=T_k$ and for every $m\in[r]$ choose $s_m$ to be the first time the clock at $x_m$ rings after time $s_{m-1}$.
Let $G_m$ be the event that $s_m-s_{m-1}\le 2t_{k+1}/r$, so 
\[\p_p(G_m)=1-\exp(-2t_{k+1}/r)\le 2t_{k+1}/r,\]
and the events $G_m$ are independent, therefore
\begin{align*}
P_k(r)& =\p_p(s_r\le T_{k+1}) 
\le \p_p\left(\sum_{m=1}^r{\bf 1}_{G_m} \ge r/2\right) \\
& \le {{r}\choose{r/2}}
\left(\frac{2t_{k+1}}{r}\right)^{r/2} 
 \le \left(\frac{4et_{k+1}}{r}\right)^{r/2}.
\end{align*}
Finally, observe that for $r\ge r_k$ we have $\dfrac{r}{t_{k+1}}\ge \dfrac{1}{6C}\dfrac{1}{q_{k}^{c}}$ so

\begin{align*}
 \p_p(F_{k+1}')  &  \le\sum_{r=r_k}^\infty \frac{1}{q_{k+1}}(4C^2)^r\left(\frac{4et_{k+1}}{r}\right)^{r/2} 
 \le\frac{1}{q_{k+1}}\sum_{r=r_k}^\infty \left(O(q_{k}^{c})\right)^{r/2}\\
& \le\frac{1}{q_{k+1}}\exp\left[-\Omega\left(\frac{1}{q_{k}}\right)^{3c}\right] 
 \le q_{k+1}/2,
\end{align*}
since $c>1$.
\end{proof}
This closes the proof of Proposition \ref{lema1}, and we are ready to prove Theorem \ref{stronger}.
\subsection{All together now}
In this last section we will finish the proof by showing that
\begin{equation}
  \p_p[\sigma_t(0)=-]\le \exp(-t^{\gamma}),
 \end{equation}
for all $t>0$.
\begin{proof}[Proof of Theorem \ref{stronger}]
By Proposition \ref{lema1}, for all times $T_k$ we already have
\[\p_p(\sigma_{T_k}(0)=-)\le\hat q_k\le q_k=\exp\left(-at_k^{1/2c}\right),\]
and $t_{k-1}/t_k=(q_{k-1}/q_{k-2})^{2c}\le c'$ for some constant $c'<1$, so, by definition, $T_k\le (1-c')^{-1}t_k$, hence
\[\p_p(\sigma_{T_k}(0)=-)\le \exp\left(-T_k^{1/3c}\right).\]

Therefore, Theorem \ref{stronger} holds for all times of the form $t=T_k$, with any $\gamma\le 1/3c$. To conclude that it holds for all $t>0$, we use the same coupling trick used in \cite{FSS02}, which consists of
comparing evolutions started from product measures with different values of $q$.

Let us rewrite $q_k,t_k,T_k$ as $q_k(q),t_k(q),T_k(q)$ because they depend on the initial $q$.
We have shown that there exists some $b>0$, such that for every $0<q\le b$, when $t=T_k(q)$ it holds that
\begin{equation}\label{forTk}
 \p_p(\sigma_{t}(0)=-)\le \exp\left(-t^{1/3c}\right).
\end{equation}

Now write $b_k=q_k(b), u_k=t_k(b)$ and $U_k=T_k(b)$, and observe that $b_k$ decreases
(so $u_k$ increases) as $k$ increases. For fixed $k$, consider the parameter $q$ decreasing continuously
from $b$ to $b_1$, so the corresponding $T_k(q)$ increasing continuously from $T_k(b)$ to
$T_k(b_1)= t_1(b_1)+\cdots+ t_k(b_1)= U_{k+1}-u_1$. 

By continuity of $T_k(q)$ and the intermediate value theorem,
any $t$ outside the union of intervals $I:=\bigcup\limits_{k\ge 1}[U_k-u_1,U_k)$ can be written as $t=T_{k(t)}(q[t])$, for some
$k(t)\ge 1$, $b_1<q[t]\le b$. 
Observe that $p=1-q\ge 1-b_1\ge 1-q[t]=:p[t]$.

Combining monotonicity and (\ref{forTk}) we get
\[\p_p(\sigma_{t}(0)=-)\le \p_{p[t]}(\sigma_{t}(0)=-)\le \exp\left(-t^{1/3c}\right),\]
and now, we have shown that the theorem holds for all $q<b_1$ and $t\notin I$.

Finally, suppose that $t\in [U_k-u_1,U_k)$ for some $k$. The key observation is the following: if $\sigma_{t}(0)=-$
and the spin at the origin does not flip between times $t$ and $U_k$ then $\sigma_{U_k}(0)=-$.
\begin{align*}
\p_p[\sigma_{t}(0)=-]&\le e^{u_1}\p_p[\sigma_{U_k}(0)=-]
 \le e^{u_1}\exp\left(-U_k^{1/3c}\right)\\
& \le \exp\left(-t^{1/4c}\right),
\end{align*}
since the probability that no flips occur at the origin from $t$ to $U_k$ is at least
$e^{-u_1}$. Thus, $\gamma=1/4c$ works for all $t\ge 0$.
\end{proof}



\section*{Acknowledgements}
The author is very thankful to Rob Morris for his  guidance and feedback throughout this project, and
Janko Gravner for his invaluable insight and suggestions on the final version of this paper.
The author would like to thank the  Instituto Nacional de Matem\'atica Pura e Aplicada (IMPA) for the time and space to create, research and write in this strong academic environment.




\bibliographystyle{plain}
\bibliography{References}
\end{document}